\documentclass[12pt]{amsart}
\usepackage{amssymb, amsmath, amsfonts, amscd}
\usepackage[dvipsnames]{xcolor}
\usepackage[raiselinks=false,colorlinks=true,citecolor=blue,urlcolor=blue,linkcolor=blue,bookmarksopen=true,pdftex]{hyperref}
\newtheorem{theorem}{Theorem}

\newtheorem{remark}[theorem]{Remark}
\newtheorem{definition}[theorem]{Definition}

\newcommand{\SL}{\textrm{SL}}

\newcommand{\Z}{{\mathbb Z}}

\newcommand{\R}{{\mathbb R}}

\begin{document}
\baselineskip=15.5pt
\title{Twisted conjugacy in fundamental groups of geometric $3$-manifolds}  
\author{Daciberg Gon\c calves}
\address{Instituto de Matem\'atica e Estat\'istica da Universidade de S\~ao Paulo\\ Departamento de Matem\'atica\\Rua do Mat\~ao, 1010 CEP 05508-090\\S\~ao Paulo-SP, Brasil}
\email{dlgoncal@ime.usp.br}
\author{Parameswaran Sankaran}
\address{Chennai Mathematical Institute , H1 SIPCOT IT Park, Siruseri, Kelambakkam, 603103, India}
\email{sankaran@cmi.ac.in}
\author{Peter Wong}
\address{Department of Mathematics, Bates College, Lewiston, Maine, USA}
\email{pwong@bates.edu}
\thanks{The first author is partially supported by Projeto Tem\'atico-FAPESP Topologia Alg\'ebrica, Geom\'etrica e Diferencial 2016/24707-4 (S\~ao
Paulo-Brazil). The first and third authors thank the IMSc (August 2018) and the CMI, Chennai (December 2019), for their support during their visits. Second and third authors thank the IME-USP, S\~ao Paulo for its support during the authors' visit in February 2019.} 

\date{\today}

\begin{abstract} A group
 $G$ has the $R_{\infty}$-property if for every $\varphi \in {\rm Aut}(G)$, there are an infinite number of $\varphi$-twisted conjugacy classes of elements in $G$. In this 
 note, we determine the $R_{\infty}$-property for $G=\pi_1(M)$ for all geometric $3$-manifolds $M$.  
\end{abstract}
\maketitle

\section{Introduction}
Let $M$ be a closed connected $n$-manifold and $f:M\to M$ a selfmap. Classical Nielsen fixed point theory is concerned with the minimal number of fixed points among all maps homotopic to $f$, i.e., the number $MF[f]:=\min_{g\sim f}\{\#({\rm Fix}g=\{x\in M \mid g(x)=x\})\}$. If $n\ge 3$, a classical theorem of Wecken asserts that $MF[f]=N(f)$, the Nielsen number of $f$. For $n=2$, the difference $MF[f]-N(f)$ can be arbitrarily large. For $n\ge 3$, the computation of $N(f)$ is a central issue but is very difficult in general. When $M$ is a Jiang-type space, for instance a generalized lens space, an orientable coset space of a compact connected Lie group, a spherical space form, or a nilmanifold, either $N(f)=0$ or $N(f)=R(f)$, the Reidemeister number of $f$. If $\varphi$ is the induced homomorphism of $f$ on $\pi_1(M)$, $R(f)=R(\varphi)$, the cardinality of the set of $\varphi$-twisted conjugacy classes of elements in $\pi_1(M)$. In such a situation, if $R(f)=\infty$ then $N(f)=0$ which implies that $f$ is homotopic to a fixed point free map. For example, for any $n\ge 5$, there exists an $n$-dimensional nilmanifold $M$ such that every homeomorphism $f:M\to M$ is isotopic to a fixed point free map \cite{GW2}. This result is a consequence of the $R_{\infty}$-property of $\pi_1(M)$ for certain nilmanifolds $M$. Recall that a group $G$ has the $R_{\infty}$-property if for every $\varphi \in {\rm Aut}(G)$, the set of orbits of the left action $\sigma \cdot \alpha \mapsto \sigma \alpha \varphi(\sigma)^{-1}$ is infinite.
It is therefore natural to ask for what families of $n$-manifolds $M$ does $\pi_1(M)$ have the $R_{\infty}$-property. For $n=3$, the Thurston-Perelman Geometrization Theorem asserts that every closed $3$-manifold is made up of finite pieces of $3$-manifolds equipped with geometries of the following eight types: (I)  $S^3$ (Spherical); (II)  $S^2\times \mathbb R$; (III)  $\mathbb E^3$ (Euclidean); (IV)  Nil; (V) $\widetilde {SL(2,\mathbb R)}$; (VI) $\mathbb H^2\times \mathbb R$; (VII)  Sol; (VIII)  $\mathbb H^3$ (Hyperbolic). 
By a {\it geometric} $3$-manifold, we mean a connected $3$-manifold equipped with a geometry from (I) - (VIII) with finite volume (see \cite{Wi}).  
It turns out that a geometric $3$-manifold is compact except in case 
of the following geometries where the manifold can be either compact or non-compact: 
$\widetilde{SL(2,\mathbb R)}, \mathbb H^2\times \mathbb R, \mathbb H^3$. 

The main objective of this note 
is to determine whether the fundamental group of a geometric $3$-manifold has the $R_{\infty}$-property. 
Leaving out the case of spherical geometry where the fundamental group is finite, our main result is the following:\\

\noindent
{\bf Main Theorem.} {\em Let $M$ be a geometric $3$-manifold with infinite fundamental group.  Then  $\pi_1(M)$ has the $R_\infty$-property when $M$ has any of the following geometries: $\widetilde{SL(2,\mathbb R)}, \mathbb H^2\times \mathbb R, \mathbb H^3$.  In the remaining cases, $\pi_1(M)$ has the $R_\infty$-property with the following exceptions:\\
(a) $S^2\times S^1$-{\em geometry}:~ $M\cong S^2\times S^1, S^2\tilde{\times} S^1, \mathbb RP^2\times S^1$,\\
(b) $\mathbb E^3$-{\em geometry}: The orientable manifolds with holonomy group $\{1\}, \mathbb Z_2,\mathbb Z_2\times \mathbb Z_2$,\\
(c) $Nil$-{\em geometry}:  
The circle bundles over the torus $S^1\times S^1$ with non-zero 
Euler class $k$, and, Seifert fibre spaces with base the sphere $S^2$ having four singular points of type $(2,1)$ and holonomy group $\mathbb Z_2$. \\
(d) $Sol$-{\em geometry}:   The manifolds having this geometry are of two kinds: $M_0=T\times I/\sim $,  where the 
boundary tori are glued via a Anosov diffeomorphism;  $M_1= E_0\sqcup E_1/\sim $, where $E_0, E_1$ are twisted $I$-bundles over the Klein bottle and their boundary tori are glued via an Anosov diffeomorphism. \\
(1) The group $\pi_1(M_0)=G=\mathbb Z^2\rtimes_\theta \mathbb Z$ where the $\mathbb Z$-action $\theta$ on $\mathbb Z^2$ is given by an Anosov matrix $A\in SL(2,\mathbb Z)$.  Then $G$ has the 
$R_\infty$-property if and only if any of the following holds: (i) $\det(A)=-1$, (ii) $A, A^{-1}\in SL(2,\mathbb Z)$ are not conjugates in $GL(2,\mathbb Z)$, (iii) $A,A^{-1} $ are conjugates in $ GL(2,\mathbb Z)$ but are {\em not}  conjugate to a matrix of the form $\bigl( \begin{smallmatrix}  r & s\\ s& u \end{smallmatrix}\bigr)$, and, furthermore,  neither $A$ nor $-A$ equals $X^p$ for some $p\in \mathbb Z$ and $X\in GL(2,\mathbb Z)$ with $\det(X)=-1$. \\
(2)  The group $\pi_1(M_1)$ has the $R_\infty$-property.  }


We have a related notion of a manifold possessing the $R_\infty$-property.   
\begin{definition}\label{rspace}
We say that a manifold $M$ has the $R_\infty$-{\em property} if, for every self-homotopy equivalence $f:M \to M$,
 the Reidemeister number of the automorphism $f_\#:\pi_1(M)\to \pi_1(M)$ is infinite. 
 \end{definition}

Note that when $M$ is an aspherical space $K(\pi,1)$, the topological and the algebraic notions of the $R_\infty$-properties coincide, that is, an aspherical  
manifold $M$ has the $R_\infty$-property if and only if the group $\pi_1(M)=\pi$ has the $R_\infty$-property.  
Other than spherical- and  $S^2\times \mathbb R$-geometries, the universal covers of the remaining geometric $3$-manifolds are diffeomorphic to 
$\R^3$ and so they are aspherical.  Since the fundamental groups of manifolds admitting spherical geometry are finite, the two notions trivially coincide.   We will show 
that in the remaining case of $S^2\times \mathbb R$-geometry also, the two notions agree.  In general, however, examples of smooth compact manifolds are 
known which have the $R_\infty$-property but their fundamental groups do not. (See the Appendix.)
We remark here that, in the case of $3$-manifolds, the Borel conjecture is known to be valid: if $M,M'$ are closed aspherical $3$-manifolds, 
any isomorphism $\phi:\pi_1(M)\to \pi_1(M')$ is induced by a {\it homeomorphism} $f:M\to M'$ (see \cite[\S2.1]{As}).

The proof of the Main Theorem will be spread over several sections, depending on the type of the geometry under consideration.  In many cases, the proof can be found or can be 
derived from results available in the literature.  But they are  
scattered in various papers and often do not specifically address the case of fundamental groups of geometric $3$-manifolds.  Specifically, 
the case of hyperbolic geometry follows from the work of Levitt and Lustig (for compact manifolds) and that of Fel'shtyn (for non-compact ones). In the cases of 
$\mathbb E^3$-, $S^2\times S^1$-, and $Sol$-geometries, the proof (for the most part) follows from the work of Gon\c{c}alves, Wong, and Zhao.   The result for $\mathbb E^3$-geometry was also obtained by Dekimpe and Penninckx, who considered  the more general case of three-dimensional crystallographic groups. 
The complete result in the case of Nil-geometry is due to Dekimpe \cite[Theorem 4.4]{De2}.   However, 
the results on $\mathbb H^2\times \mathbb R$- and $\widetilde{SL(2,\mathbb R)}$-geometries and the {\it complete} classification of manifolds admitting 
{\it Sol}-geometry whose fundamental groups do {\it not} have the $R_\infty$-property, could not be found in 
the literature.  

Our aim here is to present a coherent discussion of all the eight geometries, considering the importance of the role of the fundamental group in the study of $3$-manifolds.   For the $R_\infty$-property of fundamental groups of 
non-prime $3$-manifolds see \cite{GSW2}.

For the rest of the paper, we leave out the case of spherical geometry.

\section{Geometries $\mathbb H^3$ and $\mathbb E^3$}

\subsection{$\mathbb H^3$, the hyperbolic geometry}   The fundamental group of a compact hyperbolic $3$-manifold is known to be a (torsion-free) non-elementary word hyperbolic group. 
It follows from the main result of \cite{LL}  that these groups all have the $R_{\infty}$-property.  The fundamental group of a non-compact, finite volume hyperbolic $3$-manifold is relatively hyperbolic (with respect 
to the finite collection of fundamental groups at the cusps).  In this case, Fel'shtyn \cite{fel} has shown that such a group has the $R_\infty$-property (see also \cite{MS1} and \cite{MS2}). 
 

\subsection{$\mathbb E^3$, Euclidean or flat geometry}  In this case the $R_\infty$-property has been studied in a more general context.  Here we shall confine ourselves to the case of 
$3$-manifolds.  

 Any such manifold $M$ is a quotient $\mathbb R^3/\pi$  where $\pi$ is a torsion-free lattice  
in the group $Iso(\mathbb R^3)$ of isometries of  $\mathbb R^3$ and is finitely covered by the $3$-torus $\mathbb R^3/\mathbb Z^3$.  Thus $M$ is compact and the fundamental group $\pi$ of $M$ therefore admits a finite 
index subgroup isomorphic to 
$\mathbb Z^3$.  
 It turns out that $\pi$ has a {\it unique} maximal normal abelian group $\Gamma \cong \mathbb Z^3$. 
In particular, $\Gamma$ is characteristic in $\pi$.  It is the translation part of $\pi$.   
The group $\Phi:=\pi/\Gamma$, which is finite, is the {\it holonomy group} of $M$.   $\Phi$ acts on $\Gamma$ as automorphisms.  This is the same as the action of the deck transformation group of the covering $\mathbb R^3/\Gamma\to M$.   Thus we have an exact sequence in which $\Gamma\cong \mathbb Z^3$ is characteristic and $\Phi$, finite:

\begin{equation}  1\to \Gamma \to \pi \to \Phi\to 1.
\end{equation}

It is known that the fixed subgroup $\Gamma^\Phi$ equals the centre $Z(\pi)$ of $\pi$.    When $Z(\pi)\cong \mathbb Z$, the quotient $\pi/Z(\pi)$ is a planar crystallographic group $\Lambda$.  
Irrespective of the rank of $Z(\Gamma)$,  one has a projection of $\pi$ onto a planar 
crystallographic group $\Lambda$. Thus one has an exact sequence 
\begin{equation}
1\to Z\to \pi\to \Lambda\to 1. 
\end{equation}
However, only the case when $Z=Z(\pi)\cong\mathbb Z$ will be 
relevant for our purposes. 

When $M$ is non-orientable, it turns out that $M$ fibres over a circle with fibre the Klein bottle.  This results in an 
exact sequence 
\begin{equation} 1\to \pi_1(K)\to \pi \to \mathbb Z\to 1. \end{equation}

Using the notation of 
\cite{GWZ1}, up to diffeomorphism, there are a total of ten flat $3$-manifolds whose fundamental groups $\pi$ are listed below, 
where the first six are orientable and the remaining four are non-orientable.   We also indicate the holonomy 
group $\Phi$ and the centre $Z(\pi)$.  Whenever it is relevant for our purposes, we shall indicate the 
planar crystallographic group $\Lambda=\pi/Z(\pi)$ with $Z(\pi)\cong \mathbb Z$.   We will use 
the notation of Lyndon \cite{lyndon} for planar crystallographic groups.  

We denote the image of an element $\gamma\in \pi$ under the projection $\pi\to \Phi$ by $\bar{\gamma}$. 
\begin{enumerate}
\item [1.]  $\langle \alpha_1,\alpha_2,\alpha_3 | \ \alpha_i\alpha_j=\alpha_j\alpha_i, 1\le i,j \le 3 \rangle$ with holonomy $\Phi=\{1\}$.

\item [2.]  $\langle \alpha_1,\alpha_2,\alpha_3,t | \ \alpha_1=t^2, t\alpha_2t^{-1}=\alpha_2^{-1}, t\alpha_3t^{-1}=\alpha_3^{-1}, \alpha_i\alpha_j=\alpha_j\alpha_i, 1\le i,j \le 3 \rangle$ with holonomy $\Phi=\langle 
\bar t\rangle \cong \mathbb Z_2 $. $Z(\pi)=\langle \alpha_1\rangle$. $\Lambda=\pi/Z(\pi)\cong G_2$.

\item [3.]  $\langle \alpha_1,\alpha_2,\alpha_3,t | \ \alpha_1=t^3, t\alpha_2t^{-1}=\alpha_3, t\alpha_3t^{-1}=\alpha_2^{-1}\alpha_3^{-1}, \alpha_i\alpha_j=\alpha_j\alpha_i, 1\le i,j \le 3 \rangle$ with holonomy 
$\Phi=
\langle \bar t\rangle\cong 
\mathbb Z_3$. $Z(\pi)=\langle \alpha_1\rangle$.  $\Lambda=\pi/Z(\pi)\cong G_3$.

\item [4.]  $\langle \alpha_1,\alpha_2,\alpha_3,t | \ \alpha_1=t^4, t\alpha_2t^{-1}=\alpha_3, t\alpha_3t^{-1}=\alpha_2^{-1}, \alpha_i\alpha_j=\alpha_j\alpha_i, 1\le i,j \le 3 \rangle$ with holonomy 
$\Phi=\langle \bar t\rangle \cong \mathbb Z_4$.
$Z(\pi)=\langle \alpha_1\rangle$. $\Lambda=\pi/Z(\pi)\cong G_4$. $\Lambda=\pi/Z(\pi)\cong G_4$.

\item [5.]  $\langle \alpha_1,\alpha_2,\alpha_3,t | \ \alpha_1=t^6, t\alpha_2t^{-1}=\alpha_3, t\alpha_3t^{-1}=\alpha_2^{-1}\alpha_3, \alpha_i\alpha_j=\alpha_j\alpha_i, 1\le i,j \le 3 \rangle$ with holonomy 
$\Phi=\langle \bar t\rangle \cong \mathbb Z_6$.  $Z(\pi)=\langle \alpha_1\rangle$.  $\Lambda =\pi/Z(\pi)\cong G_6$.

\item [6.]  $\langle \alpha_1,\alpha_2,\alpha_3,t_1,t_2,t_3 | \ \alpha_1\alpha_3=t_3t_2t_1, \alpha_i=t_i^2, t_i\alpha_jt_i^{-1}=\alpha_j^{-1} \text{~for~}i\ne j, \alpha_i\alpha_j=\alpha_j\alpha_i, 1\le i,j \le 3 \rangle$ with holonomy $\Phi=\mathbb Z_2\times \mathbb Z_2$, generated by $\bar{t}_1,\bar{t}_2$.  $Z(\pi)=\{1\}$.

\item [7.]  $\pi_1(K) \times \mathbb Z=\langle \alpha,\beta|\beta\alpha\beta^{-1}=\alpha^{-1}\rangle \times \langle t \rangle$ where $K$ is the Klein bottle, with holonomy $\Phi=\mathbb Z_2$, generated by 
$\bar{\beta}$.   $Z(\pi)=gp \langle \beta^2, t\rangle\cong \mathbb Z^2$.  

\item [8.]  $\langle \alpha,\beta,t|\beta\alpha\beta^{-1}=\alpha^{-1}, t\alpha t^{-1}=\alpha, t\beta t^{-1}=\alpha\beta \rangle$ with holonomy $\Phi=\mathbb Z_2$ generated by $\bar{\beta}$.  $Z(\pi)=\langle \beta^2\rangle 
\cong \mathbb Z$. We have $gp\langle \alpha,\beta\rangle \cong \pi_1(K)$ and so 
$\pi=\pi_1(K)\rtimes \langle t\rangle$.

\item [9.]  $\langle \alpha,\beta,t|\beta\alpha\beta^{-1}=\alpha^{-1}, t\alpha t^{-1}=\alpha, t\beta t^{-1}=\beta^{-1} \rangle$ with holonomy $\Phi=\mathbb Z_2\times \mathbb Z_2$, generated by $\bar{t},\bar{\beta}$.   $Z(\pi)=\langle  t^2\rangle\cong \mathbb Z. $  We have $gp\langle \alpha,\beta\rangle = \pi_1(K)$ and so 
$\pi\cong \pi_1(K)\rtimes \langle t\rangle$. $\pi_1(K)$ is characteristic.

\item [10.] $\langle \alpha,\beta,t|\beta\alpha\beta^{-1}=\alpha^{-1}, t\alpha t^{-1}=\alpha, t\beta t^{-1}=\alpha\beta^{-1} \rangle$ with holonomy $\Phi=\mathbb Z_2\times \mathbb Z_2$, generated by $\bar{t},
\bar{\beta}$. $Z(\pi)=\langle t^2\rangle\cong \mathbb Z$.  We have $gp\langle \alpha,\beta\rangle 
=\pi_1(K)$ and $\pi=\pi_1(K)\rtimes \langle t\rangle$.   $\pi_1(K)$ is characteristic. 
\end{enumerate}

We now state the result concerning the $R_\infty$-property of these groups. 

\begin{theorem}
Let $\pi=\pi_1(M)$ where $M$ is a compact flat $3$-manifold. Then $\pi$ has the $R_\infty$-property if $\pi$ is isomorphic to one of the groups (3), (4), (5), (7), (8), (9) or (10).  
In the case when $\pi$ is isomorphic to the groups (1), (2), or (6), $M$ admits self-homeomorphisms with finite Reidemeister numbers. 
\end{theorem}   
 We merely outline the method of proof here, referring the reader to relevant papers for detailed proofs.  In case (1), the manifold $M$ is a torus and the assertion is well-known. 
 
 For Cases (4) and (5),  we use the exact sequence (2).  In these cases, $Z=Z(\pi)$ is characteristic.
Since $\pi$ projects onto a two dimensional crystallographic group $\Lambda=\pi/Z(\pi)$, which is isomorphic to 
$G_4$ and $G_6$ respectively, the  $R_{\infty}$-property of $\pi$ follows from the $R_\infty$-property of $\Lambda$ by \cite{GW3}.   

For Case (7),  $\pi=\pi_1(K)\times \mathbb Z$ where $K$ is the Klein bottle. This group is known to have the $R_{\infty}$-property (see  \cite[Theorem 2.4]{GW2}). For Cases (9) and (10), $\pi\cong \pi_1(K)\rtimes \mathbb Z$ where 
$\pi_1(K)$ is characteristic. For any $\varphi\in {\rm Aut}(\pi)$, the induced automorphism $\bar \varphi$ is either $id_{\mathbb Z}$ or $-id_{\mathbb Z}$. The former case yields $R(\varphi)=\infty$. In the later case, the set of twisted conjugacy classes of $\varphi'$ injects into the set of twisted conjugacy classes of $\varphi$, i.e., $\mathcal R(\varphi') \rightarrowtail \mathcal R(\varphi)$. Since $\pi_1(K)$ has the $R_{\infty}$-property, it follows that $R(\varphi)=\infty$.

For Cases (2), (3), (8)  we consider $\bar \pi=\Phi$ and 
$\Gamma=\mathbb Z^3$.  For Case (8), $\Phi=\mathbb Z_2$ and for Case (3), $\Phi=\mathbb Z_3$. 
In each of these two cases, we use the exact sequence (1). One can find a representative $g\in \pi$ of a suitable 
element of $\Phi$ such that 
$R(\iota_g\circ \varphi')=\infty$  (see \cite[\S 3.3 and  \S 4.3] {GWZ1}).   Since $\Gamma$ is characteristic and since $\Phi$ is finite,
   it follows that  $R(\varphi)=R(\iota_g\circ \varphi)=\infty$.  For Case (2), an explicit automorphism $\varphi$ was contructed with $R(\varphi)<\infty$ (see \S 7 of \cite{GWZ1}).

For the Case (6), consider an automorphism $\varphi$ with restriction $\varphi'$ on $\Gamma$ of type II and 
IV'
from \cite[Table 4.1]{GWZ1}. One can write down all the other three automorphisms using $\theta_1, \theta_2, \theta_3$ given few lines below the table. All such lifts have Reidemeister number finite for suitable values of $\epsilon, r,s ,b$. 
Choosing $\epsilon=1, r=1, s=a=0$ in the notation of \cite[Table 4.1,\S4.2]{GWZ1} we obtain 
that $\phi'=\begin{pmatrix}0 & 1&0\\0&0&-1\\1 &0&0 \end{pmatrix} $ so that $R(\phi')=|\det(I-\phi')|=|\det (I-\theta_i\phi')|=R(\theta_i\phi')=2$.  
Now by the addition formula, namely, \cite[Lemma 2.1]{GW1}, 
it follows that $R(\varphi)<\infty$.  Moreover, this group in Case (6) is the classical Hantzsche-Wendt group which is known not to have the $R_\infty$-property (see \cite{DDP}). 
   In \cite{GW5}, 
   all
    crystallographic groups of rank $2$ are classified in terms of the $R_{\infty}$-property.  Similarly, the result in \cite{De2} includes the full  classification of all crystallographic groups of rank $3$ in terms of the $R_{\infty}$-property, which certainly include the ten $3$-dimensional flat manifolds. So the result can be obtained from \cite{De2} after  identifying explicitly the flat $3$-manifolds as described in \cite{GWZ1} with the   ones as described in \cite{De2}. \\



\section{Geometries $S^2\times \mathbb R$ and ${\rm Sol}$}

\subsection{$S^2\times \mathbb R$-geometry}

In this case let us first analyze which of the fundamental groups have the $R_{\infty}$-property and which of the spaces have the   $R_{\infty}$-property. For the first question, the groups involved are: \\
(a) the fundamental group  of $S^2\times S^1$, which is $\Z$, and does not have the $R_{\infty}$-property;\\
(b) the fundamental group  of $\R P^2\times S^1$, which is $\Z_2\times \Z$,  and does not have the $R_{\infty}$-property; \\
(c) the fundamental group of $S^2 \tilde\times S^1$, which is $\Z$,   and does not have the $R_{\infty}$-property;\\
(d) the fundamental group of  $\R P^3\#\R P^3$, which  is $D_{\infty}=\Z_2*\Z_2$, and it has the $R_{\infty}$-property. 

Now we consider the question at the level of spaces. Certainly, since $\pi_1(\R P^3\#\R P^3)$ has the  
$R_{\infty}$-property, the manifold $\mathbb RP^3\#\mathbb RP^3$ has the $R_\infty$-property.

For the other three manifolds, $M=S^2\times S^1, \mathbb RP^2\times S^1, $ and $S^2\tilde{\times} S^1$, observe that 
in all these cases, they are total spaces of fibre bundles over $S^1$.  In each case, it is easy to construct a fibre preserving 
map $f:M\to M$ which induces the reflection map on the base space $ S^1$.  This  implies 
that the induced map $f_\#$ has finite Reidemeister number.

  For more details and further results  about the Nielsen theory of selfmaps  on such manifolds see   \cite{GWZ2}. \\


\subsection{Sol-geometry}    Let $M$ be a Sol $3$-manifold.  Then  $M$ is one of the two types:\\
(a) a  mapping torus of a self-homeomorphism $f:
T\to T$
of the torus $T=S^1\times S^1$ which induces in $\pi_1(T)=\mathbb Z^2$ an automorphism given by 
an Anosov matrix $A\in GL(2,\Z)$,\\
(b) the union of two twisted $I$-bundles over the Klein bottle glued along their common boundaries, which are tori, via 
an Anosov diffeomorphism. 
Such a manifold is also known as a sapphire manifold.

\subsection*{The case of a torus bundle}

   The homeomorphism type of $M$ is determined by the conjugacy class  of 
$A$ in $GL(2,\mathbb Z)$.    
Thus $M$ fibres over a circle with fibre $T$ and we have an isomorphism $G:=\pi_1(M)=\mathbb Z^2 \rtimes_A\mathbb Z$.    It turns out that the normal subgroup 
$N:=\mathbb Z^2\subset G$ is characteristic (see \cite[Lemma 2.1]{GW1}).    There are several cases to 
consider depending on the conjugacy class of $A$. Note that since $A$ is Anosov, (i.e., $|Tr(A)|>2, \det A=\pm 1$), 
the eigenvlaues of $A$ are real and neither of them equals $\pm 1$. In particular $A$ has infinite order.  

Case (a):  If $\det A=-1$, it was shown in \cite{GW4} that any automorphism of $G$ induces the identity map of the quotient 
$G/N=\mathbb Z$. Therefore $G$ has the $R_{\infty}$-property in this case. 

Case (b): We now assume that $\det A=1$.  
Examples of $A$ such that  $G=\mathbb Z^2\rtimes_A\mathbb Z$ does not have the $R_{\infty}$-property 
were given in  \cite[Example 4.3]{GW1}.   It was shown that when 
\begin{equation}
A=\begin{pmatrix}
  k^2+1 & k \\
  k & 1 
\end{pmatrix},
\end{equation}
$G$ does not have the $R_\infty$-property.  

A necessary condition for an automorphism $\phi:G\to G$ to have finite Reidemeister number is that the induced automorphism $\bar \phi$ on $G/N\cong \Z$ equals $-id$. 
Let $S$ be the matrix of $\phi|_N$.   Then $\bar \phi=-id$ if and only if $SA=A^{-1}S$. 
 Conversely, if $S\in GL(2,\mathbb Z)$ is such that $SAS^{-1}=A^{-1}$, then we obtain an automorphism $\phi$ of $G$ such that $\phi|_N$ is given by $S$ and $\bar{\phi}=-id$.  
{\it  In particular, if $A$ and $A^{-1}$ are not conjugates in $GL(2,\mathbb Z)$, then $G$ has the $R_\infty$-property.
}

From \cite[Proposition 5,8, Theorem 5.9]{GM1} we obtain that, if $A, A^{-1}\in SL(2,\mathbb Z)$ are conjugates 
in $GL(2,\mathbb Z)$, then
$A$ must be 
conjugate to a matrix of the form $A_0, B_0,$ or $C_0$ where

\begin{equation}\label{gluing_mapA0}
A_0=\begin{pmatrix}
  r & s \\
  s & u 
\end{pmatrix},
B_0=\begin{pmatrix}
  r & s \\
  t & r 
\end{pmatrix},
\textrm{~and~} 
C_0=\begin{pmatrix}
  r & s \\
  u-r & u 
\end{pmatrix}.
\end{equation}
We remark that an $A\in SL(2,\mathbb Z)$ may be conjugate to more than one of the above three types.

{\bf  Type $A_0$}:  {\it If  $A$ is conjugate to a matrix of the form $A_0$ in Equation \eqref{gluing_mapA0}, then 
there is an automorphism of the group $G=\pi_1(M) $ which has finite Reidemeister number.}

\begin{proof}   We may (and do) assume that $A=A_0$.  
 Consider the automorphism of $N=\mathbb Z^2$ 
given by the matrix   
$J=\bigl(\begin{smallmatrix}
  0 & -1 \\
  1 & 0 
\end{smallmatrix}\bigr)$.  
Since $JA_0J^{-1}=A_0^{-1}$,  $J$ extends to an automorphism $\phi$ of the group $G=N\rtimes_{A_0}\mathbb Z$.  
Then $\phi$ induces $-id$ on the quotient $G/N=\mathbb Z$.  The Reidemeister number $R(\phi)$ can then be 
calculated using a certain {\it addition formula} (see \cite[Lemma 2.1]{GW1}).   In our context, we obtain 
$R(\phi)=R(J)+R(JA)$.  Since for any $S\in GL(2,\mathbb Z)$ we have 
$R(S)= |\det(I-S)|$ when $\det(S)$ is non-zero, we obtain that $R(\phi)= |\det(I-J)|+|\det(I-A_0J)|=4$.  
\end{proof}

It was shown in \cite[Theorem 2.2]{GW4} that the Nielsen number of any homeomorphism of $M$ (where the gluing torus homeomorphism corresponds to $A_0$) is equal to either $0$ or $4$ and that both possibilities do occur.

We shall now treat simultaneously the remaining cases when $A$ is of type $B_0$ or $C_0$. 
 Recall that $\det(A)=1$.  We say that an element $A_1\in GL(2,\mathbb Z)$ is a {\it primitive root} of $A$ if 
$\delta A=A_1^m$ 
with $m\ge 1$ maximum and $\delta\in \{1,-1\}$.

{\bf Types $B_0 $ and $C_0$:}  {\it Let 
$A$ be  conjugate to an Anosov matrix of the form $B_0$ or $C_0$ in Equation \eqref{gluing_mapA0}.  
Then:  the  group $G=\pi_1(M)=\mathbb Z^2\rtimes_A \mathbb Z$  has the $R_{\infty}$-property if and only if any primitive root of $A$ has determinant $+1$.}\\
\noindent
 {\it Proof.}
We assume, as we may,  that $A=\bigl( \begin{smallmatrix}  r & s\\t& r  \end{smallmatrix} \bigr)$ when it is of type $B_0$  and $A=\bigl( \begin{smallmatrix}  r & s\\u-r& r  \end{smallmatrix} \bigr)$ when it is of type 
$C_0$.   Also, since the case when $A$ is of type $A_0$ had already been considered, 
we assume that $t\ne s$.

Let $\phi\in Aut(G)$. Let $S\in GL(2,\mathbb Z)$ be the matrix of the automorphism $\phi|_N$, where $N= \mathbb Z^2$.  Recall that $N$ is characteristic in $G$.  
 If $\bar{\phi}$ induces the identity on $G/N\cong \mathbb Z$, then $R(\phi)=\infty$. So assume that 
$\bar\phi=-id$.   Then $S$ satisfies the equation $SAS^{-1}=A^{-1}$.

Consider the group $K(A)=\{X\in SL(2,\mathbb Z)\mid XAX^{-1}=\delta A, \delta \in \{1,-1\}\}$. 
Note that the centralizer $Z(A)\subset SL(2,\mathbb Z)$ is subgroup of $K(A)$ of index at most $2$.
Since $A\in SL(2,\mathbb Z)$ is Anosov, it follows that 
its centralizer $Z(A)\subset GL(2,\mathbb Z)$ is virtually  infinite cyclic.  In fact, the image of $K(A)$ 
in $PSL(2,\mathbb Z)=SL(2,\mathbb Z)/\{I_2,-I_2\}$ under the natural projection equals the centralizer $Z(\bar A)$ where 
 $\bar A:=\{\pm A\}\in PSL(2,\mathbb Z)$.   Using the length function associated to the free product 
$PSL(2,\mathbb Z)\cong \mathbb Z_2*\mathbb Z_3$, we see that $\bar A$ has a {\it unique} primitive 
root $\bar A_1$ in $PSL(2,\mathbb Z)$. That is,  there exists a unique $ \bar A_1\in PSL(2,\mathbb Z)$ such that  
$\bar A=\bar A_1^k$ with $k\ge 1$ largest.  Moreover $Z(\bar A)$ equals $\langle \bar A_1\rangle \cong \mathbb Z$. 
Then $A_1\in K(A)$ is such that $A_1^k=\delta _0A$, $\delta_0\in \{1,-1\}$.  The only other element in $SL(2,\mathbb Z)$ to have this 
property in $K(A)$ is $-A_1$.    Further, $K(A)$ equals $gp\langle A_1, -I_2\rangle \cong \mathbb Z\times \mathbb Z_2$.

\noindent {\it ``If" part:}  Suppose that every primitive root $A'$  of $A$ in $GL(2,\mathbb Z)$ has determinant $+1$.   Then $A'=A_1 $ or $-A_1$.  
We set $S_0:=\bigl(\begin{smallmatrix} 1&0\\0&-1\end{smallmatrix}\bigr)$ when $A$ is of type $B_0$ 
and $S_0:=\bigl(\begin{smallmatrix} 1&1\\0&-1\end{smallmatrix}\bigr)$ when $A$ is of type $C_0$.  Note that 
$S_0^2=I_2$ and $S_0AS_0^{-1}=A^{-1}$ in each type.  We shall show that the same holds for the primitive root 
$A_1$.

Suppose that $YAY^{-1}=\delta A^\epsilon , \delta,\epsilon \in \{1,-1\}$.  Then $S_0Y. A.Y^{-1}S_0^{-1}=\delta S_0A^\epsilon S_0^{-1}=\delta A^{-\epsilon}$. 
Hence the group $N(A):=\{S\in GL(2,\mathbb Z)\mid SAS^{-1}=\delta A^{\epsilon},~\delta,\epsilon
\in \{1,-1\}\}$ contains $K(A)$ as an index-$2$ subgroup since $A_1$ is a primitive root of $A$ in $GL(2,\mathbb Z)$. 
  Moreover, since $\pm A_1$ are the only primitive roots of $A$ in $GL(2,\mathbb Z)$, we have  
$N(A)=K(A)\rtimes \langle S_0\rangle $ and so, any 
element of $N(A)$ can be expressed as $\epsilon A_1^pS_0^j, \epsilon \in \{1,-1\}, p\in \mathbb Z, j=0,1.$   Thus $N(A)$ 
acts on $K(A)$ via conjugation and also on $Z(\bar {A})$.  In view of the uniqueness of the primitive root of 
$\bar A$, we see that, $SA_1S^{-1}=\delta_1 A_1^\epsilon$  for all $S\in N(A)$, where 
$\delta_1,\epsilon \in \{1,-1\}$.  We have $A^{-1}= S_0 AS_0^{-1}=S_0A_1^kS_0^{-1}=\delta^k _1A^\epsilon$.   
This implies 
that $\epsilon =-1$ and $\delta^k_1=1$.  If $k$ is odd, we have $\delta_1=1$ and $S_0A_1S_0^{-1}=A_1^{-1}$ and so 
$A_1$ is of type $B_0$ (resp. $C_0$) depending on the matrix $S_0$.  Since $A_1^k=\delta_0 A$, if $k$ is odd, we replace $A_1$ by $-A_1$ 
so that $(-A_1)^k=A$ resulting in $\delta_0=1$.  Again $-A$ is of type $B_0$ (resp. type $C_0$) and the 
same holds for $A_1$ as well.   It can be shown (by induction) that if $X$ is of type $B_0$ (resp. $C_0$), the same is 
true of $X^p$ for all non-zero integers $p$.  Therefore if  $SAS^{-1}=A^{-1}$, then $S, AS \in N(A)$ are of the form 
$\eta A_1^q S_0$ for some $\eta\in \{1, -1\}, q\in \mathbb Z$.


To complete the proof that $R(\phi)=\infty$, 
we now apply the addition formula $R(\phi)=R(S)+R(AS)$ where $SAS^{-1}=A^{-1}$.  First consider the 
case when $A$ is of type $B_0$.  We have $S=\eta A_1^p S_0, \eta\in \{1.-1\}$ where  
$S_0:=\bigl(\begin{smallmatrix} 1&0\\0&-1\end{smallmatrix}\bigr)$.  If $p=0$, then $S=\eta S_0$ and $I-\eta S_0$ 
is singular. So $R(S)=\infty$.  
Suppose that $p\ne 0$.  Write $A_1=
\bigl(\begin{smallmatrix} x&y\\ z & x\end{smallmatrix}\bigr)\in SL(2,\mathbb Z).$  Then $S=\eta \bigl(\begin{smallmatrix}  x&-y\\ z &-x \end{smallmatrix}\bigr)$.    So $\det(I_2-S)=   1-(x^2-yz)=0$ since $A_1\in SL(2,\mathbb Z)$.    Therefore $R(\phi)=\infty$. 
The same argument applies in the case of Type $C_0$ to yield $R(S)=\infty$ and so $R(\phi)=\infty$.

{\it ``Only if" part:} Suppose that $\delta A=X^m_0$ for some $X_0\in GL(2,\mathbb Z), m\ge 1,$ with $\det(X_0)=-1$, 
$\delta\in \{1,-1\}$. Then $m$ is even;     
write $m=2n$.  
Let $X_1=X^2_0$ so that $X _1^n=\delta A$.  By what has been shown already, $X_1$ is of type $B_0$ or $C_0$.  We claim 
that $X_0$ is of the same type as $A$---$B_0$ or $C_0$.    To see this, write $X_0=\bigl(\begin{smallmatrix} 
p & q \\ r &s
\end{smallmatrix}\bigr)$.  Then $X_1^2=\bigl(\begin{smallmatrix} 
p^2+qr & q(p+s) \\ r(p+s) &s^2+qr
\end{smallmatrix}\bigr)$.  If $A$ is of type $B_0$, then we must have $p^2=q^2$.  If $p=-q$, then $X_1$---and hence $A$---would be diagonal. So $p=q$ and $X_0$ is of type $B_0$.  The proof that $X_0$ is of type $C_0$ when $A$ is, is 
similar and omitted.  

Consider the automorphism $\phi$ of $G$ whose restriction to $N=\mathbb Z^2$ is given by $S=S_0X_0$.  
Then $SAS^{-1}= S_0 A S_0^{-1}=A^{-1}$.  So $\bar \phi=-id$.    Using the addition formula 
we obtain that $R(\phi)=R(S)+R(AS)$.  Proceeding as before, we see that $R(S)=2=R(AS)$.  

 In summary, we have shown that:
 {\it The group $\pi_1(M)=G=\mathbb Z^2\rtimes_A\mathbb Z$ has the $R_\infty$-property if (i)  $\det A=-1$, (ii) $A\in SL(2,\mathbb Z)$ is not conjugate to $A^{-1}$, and, (iii) $A\in SL(2,\mathbb Z)$ is of the type 
 $B_0$ or $C_0$, and, neither $A$ nor $-A$ is in the cyclic group generated by an element 
of $GL(2,\mathbb Z)$ having determinant equal to $-1$.    If $A\in SL(2,\mathbb Z)$ is of type $A_0$,  then 
$\pi_1(M)$ does not have the $R_\infty$-property.}

This completes the proof of part (d) of Main Theorem for torus bundles case.





\subsection*{The case of a sapphire manifold}  Recall that a {\it sapphire} is a $3$-manifold obtained from two orientable $3$-manifolds which are 
twisted $I$-bundles over a Klein bottle glued along their boundary tori. A sapphire which is not a torus bundle over the circle admits Sol-geometry 
when the gluing map is an Anosov homeomorphism.   If $M$ is a sapphire which is not a torus bundle, it is double 
covered by a torus bundle $\tilde M$ with fundamental group $L:=\mathbb Z^2\rtimes_A \mathbb Z$ where $A$ is a hyperbolic 
matrix in $\SL(2,\mathbb Z)$.    Moreover the index $2$ subgroup $L$ is characteristic 
in $G:=\pi_1(M)$ (see \cite[Lemma 3.1]{GW4}).   So, any automorphism $\phi$ of $G$ restricts to an automorphism $\phi'$ of $L$.
It follows that $R(\phi)=\infty$ if $R(\phi')=\infty$.  
  If $\phi'$ 
induces $id$ on the quotient $L/N\cong \mathbb Z$ where $N\cong \mathbb Z^2$ is the characteristic subgroup of $L$ corresponding to the 
fundamental group of the torus fibre in $\tilde M$, then $R(\phi')=\infty$.   In case $\phi'$ induces $-id$, choose an element $\alpha\in G\setminus  L$.   Denote by $\iota_\alpha$ the inner conjugation by $\alpha$.  
Then the automorphism $\iota_\alpha\circ \phi=:\psi$ has the same Reidemeister number as $\phi$.  Moreover, 
using the fact that $M$ does not admit an orientation reversing homeomorphism, it can be shown that, when $\phi'$ induces $-id$ on $L/N$, then 
$\iota_\alpha\circ \phi'$ induces $id$ on $\phi$.  Again we are led to the conclusion that $R(\phi)=\infty$.  See \cite[Theorems 3.4 and  4.2]{GW4} 
for a more geometric proof.

\section{Geometries $\mathbb H^2\times \mathbb R$ and $\widetilde {SL(2,\mathbb R)}$}

In this section, we focus on those geometric $3$-manifolds that are finitely covered by $3$-manifolds that are $S^1$-bundles over hyperbolic surfaces.

Let $\Gamma=\pi_1(M)$ where $M$ admits either a $\mathbb H^2\times \mathbb R$-geometry or an 
$\widetilde{SL}(2,\mathbb R)$-geometry.  Then $M$ admits a finite cover $\hat{M}\to M$ such that 
$\hat{M}$ fibres over an orientable finite volume hyperbolic surface $\Sigma$ with fibre $S^1$. Thus $\chi(\Sigma)<0$ 
and $\Sigma$ is compact if and only if $M$ is. 
 In the case of 
$\mathbb H^2\times \mathbb R$-geometry, the $S^1$-bundle may be assumed to be the 
product bundle $\Sigma\times S^1$.

 Thus we have an exact sequence  
\[1\to Z\to \Lambda\stackrel{\eta}{\to} \pi_1(\Sigma)\to 1\eqno(2) \]
where $Z=\pi_1(S^1)\cong \mathbb Z$.  In the case of $\mathbb H^2\times \mathbb R$-geometry, 
since $\hat{M}\cong \Sigma\times S^1$, $Z$ equals the centre of $\Lambda$.  

Suppose that $M$ has the $\widetilde{SL(2,\mathbb R)}$-geometry.   Let $\mathcal Z$ denote the centre of 
$\Lambda$.   Since $\pi_1(\Sigma)$ is a nonabelian free group or a higher genus surface group, its centre is trivial. 
It follows that $\mathcal Z\subset Z$. 
 We claim that $\mathcal Z$ is non-trivial and hence infinite cyclic.  To get a contradiction, suppose that 
 $\mathcal Z$ is trivial.  Then $\Lambda$ 
 maps isomorphically onto its image under the projection $p:\widetilde{SL(2,\mathbb R)}\to PSL(2,\mathbb R)$ in 
view of the fact that $\ker(p)\cong \mathbb Z$ is the centre of $\widetilde{SL(2,\mathbb R)}$. 
So 
$\hat{M}=\widetilde{SL(2,\mathbb R)}/\Lambda\to PSL(2,\mathbb R)/p(\Lambda)$ is an infinite covering projection. 
This contradicts the finiteness of the volume of $\hat{M}$. Hence our claim.

\begin{theorem}
Let $M$ be a $3$-manifold which admits a $\mathbb H^2\times \mathbb R$-geometry or $\widetilde{SL(2,\mathbb R)}$-geometry.  Then $\pi_1(M)$ has the $R_\infty$-property.
\end{theorem}
\begin{proof}
By the above discussion, the group $\Gamma=\pi_1(M)$ has a finite index subgroup $\Gamma_0$ whose centre 
$\mathcal Z$ is an infinite cyclic group and $\Gamma_0/\mathcal Z$ is either a nonabelian free group of finite rank 
or the fundamental group of a closed surface $\Sigma_0$ of genus $g\ge 2$.   Thus $\Gamma_0/\mathcal Z$ 
is a non-elementary hyperbolic group and so has the $R_\infty$-property.   Hence, it follows from 
\cite[Lemma 1.1]{GW2} that $\Gamma_0$ has the $R_\infty$-property.  The same argument shows that 
any finite index subgroup of $\Gamma_0$ also has the $R_\infty$-property.  

Since $\Gamma$ is finitely generated and since $\Gamma_0$ has finite index in $\Gamma$, it follows that 
there is a finite index subgroup $K\subset \Gamma_0$ that is characteristic in $\Gamma$. (For example, 
we may take $K$ to be the intersection of all subgroups of $\Gamma$ having index equal to the index of $\Gamma_0$ 
in $\Gamma$.)  Then $K$ has the $R_\infty$-property and so, by \cite[Lemma 1.1]{GW2}, $\Gamma$ also 
has the $R_\infty$-property.   
\end{proof}

\medskip

\section{Nil-geometry} 

The closed $3-$manifolds which admit Nil-geometry are listed as 
the infranilmanifolds $M$ of dimension $3$,  following Dekimpe  \cite[Theorem 6.5.5, Chapter 6, p. 154]{De1}.   The last column denotes the holonomy group $F$. The group $F$ may be described as the $\pi_1(M)/\Gamma$ where 
$\Gamma$ is the unique maximal nilpotent normal subgroup of $\pi_1(M)$.   The manifold $M$ is then covered 
by the compact nilmanifold $\tilde M$ with covering group $F$.  Since $M$ has Nil-geometry, it is understood that $\Gamma$ is not abelian; equivalently $\tilde M$ is not the torus.  

\begin{center}
\begin{tabular}{|c||c|c|}
\hline
Type & Set of Seifert Invariants& $F$\\
\hline\hline
i &$M_1(k)=\{k,(o_1,1);\}$&1\\
\hline
ii &$M_2(k)=\{k-2,(o_1,0);(2,1),(2,1),(2,1),(2,1)\}$&$\Z_2$\\
\hline
iii &$M_3(k)=\{k,(n_2,2);$\}&$\Z_2$\\
\hline
iv &$M_4(k)=\{k-1,(n_2,1);(2,1),(2,1)\}$&$\Z_2\times\Z_2$\\
\hline
{v}& $M_5(k)=\{k-2,(o_1,0);(4,3),(4,3),(2,1)\}$&
{$\Z_4$}\\
&$M_6(k)=\{k-1,(o_1,0);(4,1),(4,1),(2,1)\}$&\\
&$M_7(k)=\{k-2,(o_1,0);(4,3),(4,1),(2,1)\}$&\\
\hline
{vi}& $M_8(k)=\{k-2,(o_1,0);(3,2),(3,2),(3,2)\}$&
{$\Z_3$}\\
&$M_9(k)=\{k-1,(o_1,0);(3,1),(3,1),(3,1)\}$&\\
&$M_{10}(k)=\{k-2,(o_1,0);(3,2),(3,1),(3,1)\}$&\\
&$M_{11}(k)=\{k-2,(o_1,0);(3,2),(3,2),(3,1)\}$&\\
\hline
{vii}& $M_{12}(k)=\{k-2,(o_1,0);(6,5),(3,2),(2,1)\}$&
{$\Z_6$}\\
&$M_{13}(k)=\{k-1,(o_1,0);(6,1),(3,1),(2,1)\}$&\\
&$M_{14}(k)=\{k-2,(o_1,0);(6,1),(3,2),(2,1)\}$&\\
&$M_{15}(k)=\{k-2,(o_1,0);(6,5),(3,1),(2,1)\}$&\\
\hline
\end{tabular}
\end{center}

\medskip
In this table the integer $k$ is assumed to be strictly bigger than $0$. The case of manifolds with 
the same Seifert invariants but with $k\leq 0$ are either flat manifolds  or they represent  manifolds which 
are  already homeomorphic to one  in the Table for $k>0$. When the manifold is not flat, one passes from one   
family of invariants with $k<0$ to another with $k>0$ by changing the orientation (see \cite{Se}).

The complete answer to the question about $R_{\infty}$-property of the manifolds above is given by \cite[Theorem 4.4]{De2} which in terms of the 
table above says: A closed infranil-manifold has the $R_{\infty}$-property, if and only if it does not belong to the first two lines. Furthermore, 
for the cases of the manifolds of types (i) and (ii), namely the ones which do not have the $R_{\infty}$-property, the 
Reidemeister spectrum is described in \cite[section 5]{Te}.


\begin{remark} The fundamental groups of the Seifert manifolds have a presentation given by \cite[Chapter 5 section 5.3]{Or}. It follows from such presentation that we have 
 a short exact sequence 
${\displaystyle 1 \rightarrow  \langle h\rangle \rightarrow \pi _{1}(M)\rightarrow \pi _{1}(B)\rightarrow 1} $, where $h$ is a generator which corresponds to the regular fibre. 
From  \cite[Chapter 5, \S 5.3, Lemma 1]{Or}, it follows that this short exact sequence is characteristic with respect to automorphisms of $\pi_1(M)$. 
The results obtained for the 3-manifolds may help to study 
the $R_{\infty}$-property for the groups $\pi_1(B)$.  The groups $\pi_1(B)$ are most often 
Fuchsian groups.  
 \end{remark}

 \section{Appendix}

The purpose of the appendix is to show that the notions of a space $X$ having the $R_{\infty}$-property and the group $\pi_1(X)$ having the
$R_{\infty}$-property are not equivalent.

Certainly if $\pi_1(X)$ has the $R_{\infty}$-property, then the space
$X$ has the $R_{\infty}$-property. We now provide an example where $X$ has the $R_{\infty}$-property but
its fundamental group $\pi_1(X)$ does not have.

Let $Y=S^3\times S^3$ and $h: Y \to Y$ a homeomorphism  which induces on the homology group $H_3(S^3\times S^3; \Z)$
a homomorphism given by a matrix

\begin{equation}\label{gluing_map}
A=\begin{pmatrix}
	r & s \\
	t & u
\end{pmatrix}
\end{equation}
in $GL(2, \Z)$.

Let $X$ be the mapping torus of the  homeomorphism $h$ of $S^3\times S^3$, so 
$X$ fibres over the  circle $S^1$ with fibre   $S^3\times S^3$. Let $f:X \to X$ be a homotopy equivalence. Then 
the map $f$ can be deformed to a fibre-preserving map. We assume that $f$ itself is fibre preserving.
The induced map $\bar f$ on the base $S^1$ is either of degree $1$ or $-1$. The induced isomorphism on $H_3(X)\cong \Z^2$ is a matrix $B$ such that: If the degree of
$\bar f$ is $-1$ then we must have
$BAB^{-1}=A^{-1}$. Let $A$ be a matrix where such $B$ does not exist (see \S3.2 above). Then any homotopy equivalence will induce a map of degree $1$ on $S^1$ and hence identity on $\pi_1(X)$. So
$R(f)={\infty}$, but certainly $\Z=\pi_1(X)=\pi_1(S^1)$ does not have the $R_{\infty}$-property. So it suffices to choose  a matrix $A$ where $B$ does not exist and determine a homeomorphism $h_{A}$.
We can take $A$ any matrix such that $\det A=-1$ or one of the matrices from \cite[Example 4.4]{GW1}. In details, for the former case let 
\begin{equation}\label{gluing_mapI}
A=\begin{pmatrix}
	1 & 1 \\
	2 & 1
\end{pmatrix}
\end{equation}
so that 
\begin{equation}\label{gluing_mapI}
A^{-1}=\begin{pmatrix}
	-1 & 1 \\
	2 & -1
\end{pmatrix}.
\end{equation}

Define $h_A(q_1,q_2)=(q_1q_2, q_1^2q_2))$ and  $h_{A^{-1}}(q_1,q_2)=(q_2q_1^{-1}, q_1q_2^{-1}q_1)$.\\
It is straightforward to verify that $h_{A}\circ h_{A^{-1}}=h_{A^{-1}}\circ h_{A}=Id_{S^3\times S^3}$. So we can apply our construction for 
$h_{A}$ and we obtain the result.

For a matrix in the latter case (as in \cite[Example 4.4]{GW1}), let 
\begin{equation}\label{gluing_mapI}
A=\begin{pmatrix}
	4 & 1 \\
	3 & 1
\end{pmatrix}
\end{equation}
so that 
\begin{equation}\label{gluing_mapI}
A^{-1}=\begin{pmatrix}
	1 & -1 \\
	-3 & 4
\end{pmatrix}.
\end{equation}

Define $h_A(q_1,q_2)=(q_1^4q_2, q_1^3q_2)$ and  $h_{A^{-1}}(q_1,q_2)=(q_1q_2^{-1}, q_2q_1^{-1}q_2q_1^{-1}q_2q_1^{-1}q_2)$.\\
It is straightforward to verify that $h_{A}\circ h_{A^{-1}}=h_{A^{-1}}\circ h_{A}=Id_{S^3\times S^3}$. So we can apply our construction for 
$h_{A}$ and we obtain the result.

\noindent
{\bf Acknowledgments:}  The authors grateful to the referee for his/her many valuable comments and for drawing our attention 
to some inaccuracies in an earlier version of this paper.

\end{document}